\documentclass[12pt]{amsart}
\usepackage[dvips]{graphicx}
\usepackage{amsmath,graphics}
\usepackage{amsfonts,amssymb}
\usepackage{xypic}
\usepackage{comment}
\usepackage{color}
\specialcomment{proofc}{}{}
\includecomment{proofc}
\theoremstyle{plain}
\newtheorem*{theorem*}{Theorem}
\newtheorem*{lemma*} {Lemma}
\newtheorem*{corollary*} {Corollary}
\newtheorem*{proposition*}{Proposition}
\newtheorem*{conjecture*}{Conjecture}
\newtheorem{theorem}{Theorem}[section]
\newtheorem{lemma}[theorem]{Lemma}
\newtheorem*{theorem1*}{Theorem 1}
\newtheorem*{theorem2*}{Corollary 2}
\newtheorem*{theorem3*}{Corollary 3}
\newtheorem*{theorem4*}{Proposition 4}
\newtheorem{corollary}[theorem]{Corollary}
\newtheorem{proposition}[theorem]{Proposition}

\theoremstyle{remark}
\newtheorem*{remark}{Remark}
\newtheorem*{remarks}{Remarks}
\newtheorem*{definition}{Definition}

\newtheorem*{example*}{Example}

\theoremstyle{definition}

\def\op{\operatorname}

\def\G{\Gamma}

   \def\Z{\Bbb{Z}} \def\R{\Bbb{R}}  
    
 \def\a{\alpha} \def\b{\beta}  \def\bp{\begin{pmatrix}}
 \def\ep{\end{pmatrix}} \def\bn{\begin{enumerate}} 
   \def\en{\end{enumerate}}
\def\ba{\begin{array}} \def\ea{\end{array}}  
 \def\S{\Sigma}  \def\a{\alpha} \def\b{\beta}

\def\be{\begin{equation}} \def\ee{\end{equation}}

\def\op{\operatorname}

\def\mod1{\mbox{Mod}(F^1)}
\def\c{\chi}

\begin{document}
\title[BNS invariants and algebraic fibrations of group extensions]{BNS invariants and  algebraic fibrations of group extensions}

\author{Stefan Friedl}
\address{Fakult\"at f\"ur Mathematik, Universit\"at Regensburg, Germany}
\email{sfriedl@gmail.com}

\author{Stefano Vidussi}
\address{Department of Mathematics, University of California,
Riverside, CA 92521, USA} \email{svidussi@ucr.edu}

\begin{abstract} Let $G$ be a finitely generated group that can be written as an extension \[  1 \longrightarrow K \stackrel{i}{\longrightarrow}  G \stackrel{f}{\longrightarrow} \G \longrightarrow 1 \] where $K$ is a finitely generated group. By a study of the BNS invariants we prove that if $b_1(G) > b_1(\G) > 0$, then $G$ algebraically fibers, i.e. admits an epimorphism to $\Z$ with finitely generated kernel. An interesting case of this occurrence is when $G$ is the fundamental group of a surface bundle over a surface $F \hookrightarrow X \rightarrow B$ with \textit{Albanese dimension} $a(X) = 2$. As an application, we show that if $X$ has virtual Albanese dimension $va(X) = 2$ and base and fiber have genus greater that  $1$, $G$ is noncoherent. This answers for a broad class of bundles a question of J. Hillman (\cite[Question 11(4)]{Hil15}). Finally, we show that there exist surface bundles over a surface whose BNS invariants 
have f a structure that differs from that of Kodaira fibrations, determined by T. Delzant.
  
\end{abstract}

\maketitle

\section{Introduction and Main Results}

Throughout this introduction let $G$ be a finitely generated group. We say that $G$ \textit{algebraically fibers} if there exists an epimorphism $\phi\colon  G \to \Z$ with finitely generated kernel. The study of algebraic fibrations of groups is closely related to the study of the Bieri--Neumann--Strebel invariant of $G$. This is an open subset $\Sigma^{1}(G) \subset S(G)$, where \[ S(G) = (H^{1}(G;\R) \setminus \{0\})/\R_{+}\] is the sphere of \textit{characters} of $G$. Its complement $\S^{1}(G)^{c} = S(G) \setminus \S^1(G)$ is referred to as the set of \textit{exceptional} characters.  (We refer to \cite{BNS87,St13} for definitions and properties used here.) The relation between algebraic fibrations and the BNS invariant can be described as follows: If we denote by $[\phi] \in S(G)$ the character associated with $\phi\colon G\to \Z$, thought of as an element of $H^{1}(G;\Z)  \subset H^{1}(G;\R)$, then it follows from \cite[Corollary~4.2]{BNS87} that $\phi$ algebraically fibers if and only if both $[\phi], [-\phi]$ belong to $\Sigma^{1}(G)$.

Assume that $G$ is a group extension of the form
\begin{equation} \label{eq:ses}  1 \longrightarrow K \stackrel{i}{\longrightarrow}  G \stackrel{f}{\longrightarrow} \G \longrightarrow 1, \end{equation} where $K$ is finitely generated. The study of BNS invariants of group extensions is quite challenging (see e.g.\ \cite{KoWo15}). We collect some basic facts. To start, since $K$ is finitely generated it follows from \cite[Proposition~A.4.5]{St13} that  a character $\chi \in S(\G)$ belongs to $\S^{1}(\G)$ if and only if its pull back $f^{*}\chi \in \S^1(G)$. This entails that if $\G$ algebraically fibers, then so does $G$.  When the extension is trivial, i.e.\ when $G=K\times \Gamma$, much is known (see e.g.\ \cite{St13}) and the sets of exceptional characters are determined by the equality
\begin{equation} \label{eq:sigmaprod} \S^1(K \times \G)^{c} = \S^{1}(K)^{c} \cup \S^{1}(\G)^{c} \end{equation} hence, as long as both $b_1(\G),b_1(K) > 0$ the trivial extension will always have nonempty BNS invariant. 

This makes it reasonable to expect that any extension as in (\ref{eq:ses}) with $b_1(\G) > 0$ satisfies $\S^1(G) \neq \emptyset$, unless $f^{*}\colon H^{1}(\G;\R) \to H^1(G;\R)$ is an isomorphism.

Our main result is the proof that this is true, and more precisely, that  $G$ algebraically fibers.
 
\begin{theorem1*} \label{thm:main}
Let $G$ be a finitely generated group that can be written as group extension  \begin{equation} \label{eq:ses2}  1 \longrightarrow K \stackrel{i}{\longrightarrow}  G \stackrel{f}{\longrightarrow} \G \longrightarrow 1 \end{equation} where $K$ is a finitely generated group. Assume furthermore that $b_1(G) > b_1(\G) > 0$. Then $G$ algebraically fibers. 
\end{theorem1*} 

Our interest in the problem of algebraic fibrations of group extensions arose from a geometric perspective, namely the case where $G$ is the fundamental group of a surface bundle over a surface $F \hookrightarrow X \rightarrow B$ with base and fiber both of genus greater than $0$. (If the genus of the fiber is at least $2$, the condition on the Betti numbers can be phrased in terms of nonvanishing of the co-invariant homology of the fiber $H_1(F;\R)_{\G}$, see Section \ref{sec:proofs} for details).

There is one noteworthy class of surface bundles over a surface -- namely those who admit a K\"ahler structure, e.g.\ Kodaira fibrations -- where the BNS invariant is fully understood, thanks to the work of Delzant (\cite{De10}, see also \cite{FV19}). To dovetail that result with Theorem \ref{thm:main}, it is useful to introduce the following notation.

\begin{definition} Let $F \hookrightarrow X \stackrel{f}{\rightarrow} B$ be a surface bundle over a surface with base and fiber both of genus greater than $0$. Let $G := \pi_1(X)$ and, using the homotopy exact sequence of the fibration, write $G$ as the extension in (\ref{eq:ses}).
If $f^{*}\colon H^{1}(\G;\R) \to H^1(G;\R)$ is an isomorphism we say that such bundle has \textit{Albanese dimension one}, or $a(X) = 1$. Any other surface bundle will be unambiguously referred to as having \textit{Albanese dimension two}. If a surface bundle admits a finite cover that has Albanese dimension $2$, we say that $X$ has \textit{virtual} Albanese dimension $2$ and we write $va(X)=2$. \end{definition}

Our notation stems from the analogy with the class of (irregular) K\"ahler manifolds $X$  of Albanese dimension one. That condition, determined solely by the fundamental group of $X$, amounts to the existence of an irrational  \textit{Albanese pencil} (a holomorphic map $f\colon X \to B$ to a Riemann surface of positive genus  with connected fibers), obtained by restriction to the image of the Albanese map of $X$, which induces an isomorphism $f^{*}\colon H^{1}(\G;\R) \to H^1(G;\R)$ where $G = \pi_1(X)$ and $\G = \pi_{1}(B)$. Note that the definition above is consistent in the overlap of the two classes, i.e.\ surface bundles $X$ that admit a K\"ahler structure. 
In fact, our interest for the connection with the K\"ahler case arises from the fact that when $X$ is a K\"ahler surface, Delzant has shown in \cite{De10} that, as long as $a(X) = 2$, the first BNS invariant is nonempty. In particular,  $G$ algebraically fibers. 

With the notation in place, we have the following consequence of the main theorem, that we single out for its interest.

\begin{theorem2*} \label{thm:sbs} Let $F \hookrightarrow X \to B$ be a surface bundle over a surface with fiber and base both of genus greater than $0$. Assume that the Albanese dimension of $X$ is $a(X) = 2$. Then $G$ algebraically fibers. \end{theorem2*} 

Besides the interest \textit{per se} in deciding that such a group $G$ admits an algebraic fibration, we'll show that Theorem \ref{thm:sbs} entails that, as long as fiber and base have both genus greater than $1$, then $G$ is noncoherent, namely it contains a finitely generated subgroup that is not finitely presented. 

The question of coherence of the fundamental group of a surface bundle over a surface was raised by Hillman in \cite[Question 11(4)]{Hil15} (and perhaps earlier). There are two cases, as far as we know, where this group was known to be noncoherent. The first is the case of surface bundles with monodromy of types I and II in Johnson's trichotomy (\cite{Jo93}): their fundamental groups contain $F_2 \times F_2$ as subgroup. The second case  appears in \cite{FV19} where the authors show that a Kodaira fibration that has virtual Albanese dimension $2$ has fundamental group that is noncoherent. Theorem \ref{thm:sbs} allows us to proceed like in that paper to show the following:

\begin{theorem3*}  \label{cor} Let $F \hookrightarrow X \to B$ be a surface bundle over a surface with both base and fiber of genus greater than $1$. If its Albanese dimension $a(X) = 2$, then the first BNS invariant  $\Sigma^1(G)$ and the second BNSR invariant $\Sigma^2(G)$ of its fundamental group $G = \pi_1(X)$ satisfy the relation 
\[  \S^2(G) \subsetneq \S^1(G) \subsetneq S(G), \] and $G$ is noncoherent. If $va(X) = 2$, $G$ is noncoherent.
\end{theorem3*}

Here,  $\S^2(G)$ is the second \textit{Bieri--Neumann--Strebel--Renz} (BNSR) invariant of $G$, the first of a collection of refinements of the BNS invariant introduced in \cite{BR88}.

In fact, we will present two proofs of noncoherence on $G$, the second being based on an elegant construction appearing in \cite{KrWa19}, that inspired in various way the techniques employed in the present paper. 

We will finish this paper by discussing the challenge of completely determining the BNS invariant of a surface bundle over a surface. In particular, we will show the existence of a surface bundle over a surface, whose fundamental group $G$ has exceptional characters that do not arise from an epimorphism $h\colon G \to C$ to the  fundamental group of a hyperbolic (orbi)surface and finitely generated kernel. This contrasts with what happens for Kodaira fibrations (or K\"ahler manifolds). We have the following

\begin{theorem4*} There exists a surface bundle $F \hookrightarrow X \stackrel{f}{\rightarrow} B$ with base and fiber of genus greater than $1$ whose fundamental group $G$ admits an epimorphism $g\colon G \to F_2$ such that $H^1(G;\R) = f^{*}H^{1}(\G;\R) \oplus g^{*}H^{1}(F_2;\R)$ and so that the set of exceptional characters contains two disjoint spheres \[ 
 f^{*}(\S^{1}(\G)^{c}) \cup g^{*}(\S^{1}(F_2)^{c})\,\subset\, 
\Sigma^{1}(G)^{c}. \]
Moreover for any epimorphism $h\colon G \to C$ onto the fundamental group of a  hyperbolic orbisurface $C$ such that the kernel is finitely generated we have  $g^{*}(\S^{1}(F_2)^{c})\cap h^*(\S^1(C)^c)=\emptyset$.

 \end{theorem4*} 

(Note that $\S^{1}(\G)^{c} = S(\G)$ and  $\S^{1}(F_2)^{c} = S(F_2)$.)

\subsection*{Acknowledgement} The authors wish to thank the referee for carefully reading the manuscript.
SF was supported by the SFB  1085 ``higher invariants'' which is supported by the Deutsche Forschungsgemeinschaft DFG. SV was supported by the Simons Foundation Collaboration Grant For Mathematicians 524230.

\section{Proofs} \label{sec:proofs}

Before proceeding with the proofs of the results listed in the introduction, we want to discuss the meaning of the assumption $b_1(G) > b_1(\G)$ in Theorem \ref{thm:main}. Given a group extension $G$ as in (\ref{eq:ses}), the action by conjugation of $G$ on its normal subgroup $K$ induces a representation $\rho\colon G \to \operatorname{GL}(V)$  on the homology of the kernel $V = H_1(K)$, where the homology can be taken with $\Z$ or $\R$ coefficients. As the action of $K$ on its homology is trivial, this representation descends to $\G$. The Lyndon--Hochschild--Serre spectral sequence associated to  (\ref{eq:ses}) gives,  in low degree, the following exact sequence: \begin{equation} \label{eq:lhs} H_2(G) \stackrel{f}{\longrightarrow} H_2(\G) \longrightarrow H_1(K)_{\G} \stackrel{i}{\longrightarrow} H_1(G) \stackrel{f}{\longrightarrow} H_1(\G) \longrightarrow 0.  \end{equation} The image of the map $H_1(K;\R)_{\G} \to H_1(G;\R)$ measures the failure of $f \colon  H_1(G;\R) \to H_1(\G;\R)$ (or equivalently $f^{*}\colon H^{1}(\G;\R) \to H^1(G;\R)$) to be an isomorphism, i.e.\ it measures the mismatch between $b_1(G)$ and $b_1(\G)$. 
A similar sequence exists for the cohomology groups, with the role of coinvariant homology of $K$ played by the invariant cohomology group $H^{1}(K)^{\G}$. In the case where $G$ is a surface bundle with fiber of genus greater than $1$, the map $H_2(G;\R) \longrightarrow H_2(\G;\R)$ in the sequence (\ref{eq:lhs}) is surjective (see e.g.\ \cite{Mo87}), hence the condition $b_1(G) > b_1(\G)$ is equivalent to  $H_1(K;\R)_{\G} \neq \{0\}$.

In the proof of Theorem \ref{thm:main} we will use some general results on the behavior of BNS invariant for an amalgamated free product. The first is quite well--known, and appears as \cite[Lemma B1.14]{St13}. 

\begin{lemma} \label{lem:streb} Let $\Pi = \Pi_{1} *_{K} \Pi_2$ a free product with amalgamation of two finitely generated  groups along a finitely generated  subgroup $K$. Let $\c \in S(\Pi)$ be a character whose restrictions satisfy the conditions $\c_1 \in \S^1(\Pi_1)$, $\c_K \neq 0$ and $\c_2 \in \S^1(\Pi_2)$. Then $\chi \in \S^1(\Pi)$. 
\end{lemma}

The second result, instead, seems new, and it is possibly interesting \textit{per se}. The proof that we present is similar in flavor (and at times \textit{verbatim}), to the proof of Lemma \ref{lem:streb} in \cite{St13} (whose notation we follow), but requires some further work.

\begin{lemma} \label{lem:nonfib} Let $\Pi = \Pi_{1} *_{K} \Pi_2$ a free product with amalgamation of two finitely generated  groups along a finitely generated  subgroup $K$. Assume that $\Pi_{2}$ is an HNN extension  $\Pi_2 = K \rtimes \Z = \langle K,s|s k s^{-1} = f(k) \rangle$ for some automorphism $f\colon K \to K$. Let $\c \in S(\Pi)$ be a character whose restrictions satisfy the conditions $\c_{1} \in \S^{1}(\Pi_1)$, $\c_{K} \neq 0$ and $\c_{2}(s) = 0$. Then $\c \in \S^{1}(\Pi)$.
\end{lemma} 

\begin{proof}
Let $\mathcal{X}_{1}$ be a finite generating set for $\Pi_1$.  Let $\mathcal{X}_{K}$ be a finite generating set for $K$ and let $\mathcal{X}_{2} = \mathcal{X}_{K} \cup \{s\}$. Denote by $\G(\Pi_1,\mathcal{X}_1),\G(\Pi_2,\mathcal{X}_2)$  the Cayley graphs for $\Pi_1,\Pi_2$ associated to their respective generating sets.
Then $\mathcal{X} = \mathcal{X}_{1} \cup \mathcal{X}_{2}$ constitutes a finite generating set for $\Pi$, with associated Cayley graph $\G(\Pi,\mathcal{X})$.

 Recall that, by the very definition of the BNS invariant, to prove that $\c \in \S^1(\Pi)$ we need to show that the \textit{ subgraph $\G(\Pi,\mathcal{X})_{\c}$} of $\G(\Pi,\mathcal{X})$ determined by the \textit{vertices $\Pi_{\c}$  with nonnegative $\c$--value} is connected. 

Let $g \in \Pi_{\c}$; there exists a finite collection of elements $g_{1,j} \in \Pi_{1}, g_{2,j} \in \Pi_{2}, j = 1,\dots,n$ such that
\begin{equation} \label{eq:deco} g = g_{1,1}\cdot g_{2,1} \cdot g_{1,2} \cdot g_{2,2}\cdot  \dots \cdot g_{1,n} \cdot g_{2,n}. \end{equation} We will show that there exists a path in $\G(\Pi,\mathcal{X})_{\c}$ from $1$ to $g$ by induction on $n$.

Let's consider the initial case $n = 1$, i.e.\ $g = g_1 \cdot g_2$.  As $\c_{K} \neq 0$, there exist two elements $h_{1},h_{2} \in K$ such that all three \[ g'_{1} = g_1 \cdot h_1, \,\  g'_{2} = h_{1}^{-1} \cdot g_2 \cdot h_2, \,\ h_{2}\] are contained in $\Pi_{\c}$. As $\G(\Pi_1,\mathcal{X}_1)_{\c_{1}}$ is connected by the assumption that $\c_1 \in \S^1(\Pi_1)$, there exists a path $p_1 = (1,w_1)$ from $1$ to $g'_{1}$ contained in $\G(\Pi_1,\mathcal{X}_1)_{\c_1}$, where $w_1$ is a word in $\mathcal{X}_{1}^{\pm}$. 

Next, consider the element $g'_{2} = h_{1}^{-1} \cdot g_2 \cdot h_2$. As $g_2 \in \Pi_{2}$, it can be written as a word in the generating set for $K$ and the stable letter $s$. Because of the relations in $\Pi_2$, we can use the equalities 
\[ s \cdot k = f(k) \cdot s \,\ \mbox{and} \,\ s^{-1} \cdot k = f^{-1}(k) \cdot s^{-1}  \] to push powers of the stable letter to the right and rewrite $g_{2} = w(k) \cdot s^{m}$, where $w(k)$ is a word in $\mathcal{X}_{K}^{\pm}$ and $m \in \Z$. As $h_2 \in K$, we can further write \[ g'_{2} = h_1^{-1} \cdot g_2 \cdot h_2 = h_1^{-1} \cdot w(k) \cdot s^{m} \cdot h_2 = h_1^{-1} \cdot w(k) \cdot h'_2 \cdot s^{m} \] where $h'_2 =  f^{m}(h_2)$, with $f^{m}$ an iteration of $f$ or its inverse. Consider the element 
\[ h_1^{-1} \cdot w(k) \cdot h'_2 = g'_2 \cdot s^{-m} \in K.\] 
 As $\c(s) = 0$, we have  
 \[ \c_{1}(h_1^{-1} \cdot w(k) \cdot h'_2) = \c(h_1^{-1} \cdot w(k) \cdot h'_2) = \c(g'_2) \geq 0. \]
  As $K \leq \Pi_1$, there exist a path $p_2 = (1,w_2)$ from $1$ to $h_1^{-1} \cdot w(K) \cdot h'_2$ contained in $\G(\Pi_1,\mathcal{X}_1)_{\c_1}$, which again is connected by assumption.

Next, consider the path $p_3 = (1,w_3)$ in $\G(\Pi,\mathcal{X})$ from $1$ to $s^{m}$ determined by the vertices 
\[ 1,\dots,  s^{m-\sigma(m) \cdot 2}, \,\ s^{m-\sigma(m) \cdot 1},s^{m}, \] where $\sigma(m)$ is the sign of $m$. (Here, $w_3 = s^m$.)
 As  $\chi(s) = 0$, all these vertices are contained in $\Pi_{\c}$; in particular, the path $p_3$ is entirely contained in $\G(\Pi,\mathcal{X})_{\c}$. 

Finally, as $\c_2$ is well--defined over $\Pi_2$, we have $\c(k) = \c_{2}(k) = \c_{2}(f(k)) = \c(f(k))$. As a consequence, $\c_{1}(h_2) = \c_{1}(h'_2) \geq 0$, so there exists a path $p_4 = (1,w_4)$ from $1$ to $h'_2$ contained in $\G(\Pi_1,\mathcal{X}_1)_{\c_1}$. 

Concatenating the paths $p_1,p_2,p_3,p^{-1}_4$ we obtain a path $(1,w_1 w_2 w_3 w^{-1}_4)$    in $\G(\Pi,\mathcal{X})$ from $1$ to \[ g'_{1} \cdot h_1^{-1} \cdot w(k) \cdot h'_2 \cdot s^{m} \cdot h_{2}^{-1} = g_{1} \cdot h_1 \cdot h_1^{-1} \cdot w(k) \cdot s^{m} = g_1 \cdot g_2. \]
One can verify from the construction above that each vertex of the path is contained in $\Pi_{\c}$. Alternatively, one can use the valuation function on the set of words in $\mathcal{X}$ (\cite[Section A2.2]{St13}), which measures the lowest $\c$--value over vertices of a path starting at $1$, to get (using  \cite[Equations A2.7]{St13})

\begin{multline*} 
v_{\c}(w_1 w_2 w_3 w^{-1}_4) =   \\ 
  = \mbox{min} \{ v_{\c}(w_1), \c(w_1) +  v_{\c}(w_2), 
\c(w_1 w_2) + v_{\c}(w_3), \c(w_1 w_2 w_3) + v_{\c}(w^{-1}_4) \} = \\
= \mbox{min} \{0, \c(g'_1) + 0, \c(g'_1 \cdot g'_2 \cdot s^{-m}) + 0, \c(g'_1 \cdot g'_2) + v_{\c}(w_4) - \chi(w_4) \} = \\
=  \mbox{min} \{0, \c(g'_1),  \c(g'_1 \cdot g'_2), \c(g'_1 \cdot g'_2) - \c(h_2) \} = \\
= \mbox{min} \{0, \c(g'_1),  \c(g'_1 \cdot g'_2), \c(g_1 \cdot g_2) \} = 0. \\
 \end{multline*}

We now assume that the Lemma holds for $n-1$; let $g$ be like in Eq. (\ref{eq:deco}) and denote $g'$ the product of the first $2m-2$ factors. There exist elements $h_0,h_1,h_2 \in K$ contained in $\Pi_{\c}$ such that $g' \cdot h_0$, $g'_1 = h_{0}^{-1} \cdot g_{1,n} \cdot h_{1}$, $g'_2 = h_{1}^{-1} \cdot g_{2,n} \cdot h_{2}$ are contained in $\Pi_{c}$. By the inductive hypothesis, there is a path in $p' = (1,w')$ from $1$ to $g' \cdot h_0$ contained in  $\G(\Pi,\mathcal{X})_{\c}$. Moreover, there exist paths $p_i = (1,w_i), i = 1,\dots4$ (that mirror the role of the similarly--named paths for the case $n=1$) with the property that: 
\begin{itemize}
\item $p_1 = (1,w_1)$ runs from $1$ to $g'_1 \in \Pi_1$ in $\G(\Pi_1,\mathcal{X}_1)_{\c_1}$;
\item $p_2 = (1,w_2)$ runs from $1$ to $h_1^{-1} \cdot w(k) \cdot h'_2 \in K$ in $\G(\Pi_1,\mathcal{X}_1)_{\c_1}$, where $g_{2,n} = w(k) \cdot s^{m}$, and $h'_2 = f^{m}(h_2)$;
\item $p_3 = (1,w_3)$ runs from $1$ to $s^m$ in $\G(\Pi,\mathcal{X})_{\c}$;
\item $p_4 = (1,w_4)$ runs from $1$ to $h_2' \in K$ in $\G(\Pi_1,\mathcal{X}_1)_{\c_1}$.
\end{itemize}

Much as before, the concatenation of $p',p_1,p_2,p_3,p_4^{-1}$ yields a path from $1$ to $g' \cdot g_{1,n} \cdot g_{2,n} = g$ contained in $\G(\Pi,\mathcal{X})_{\c}$.
\end{proof} 

\begin{remark} Note that in Lemma \ref{lem:nonfib} we do not, nor can we, assume that $\c_{2} \in \S^{1}(\Pi_2)$; the regular characters on $\Pi$ provided by that Lemma may appear at first sight surprising. However, for instance, a careful analysis based on Equation (\ref{eq:sigmaprod}) of the BNS invariants of $F_2 \times F_2$, thought of as free amalgamated product of two copies of $F_2 \times \Z$ (whose BNS invariant is easily computed), reveals that there exist regular characters that restrict to exceptional ones on one (but not both) of the factors. In fact, the combination of Lemmata \ref{lem:streb} and \ref{lem:nonfib} provides the entirety of $\S^{1}(F_2 \times F_2)$.
\end{remark}

There is another technical lemma that guarantees the existence of a presentation of the group $\G$ that will be convenient in what follows. In order to state it, we will introduce a new definition.

\begin{definition} Let $\G$ be a finitely generated group. Denote $\op{ab}\colon \G \to H_{1}(\G;\Z)/\mbox{Tor}$ the maximal free abelian quotient map. We will say that a presentation of $\G$ with a generating set $(h_1,\dots,h_m,g_{1},\dots,g_r)$ is \textit{adjusted to $\op{ab}$} if $\op{ab}(h_i), i = 1,\dots,m$ is a basis of $H_{1}(\G;\Z)/\mbox{Tor} \cong \Z^m$ and $g_{1},\dots,g_r \in \mbox{ker} \op{ab}$.
\end{definition}

For instance, the usual presentation of a surface group is adjusted to $\op{ab}$, with $r = 0$.
The following lemma shows that such a presentation always exists.
 It is certainly well--known, but we provide a proof for completeness.

\begin{lemma} \label{lem:base} Let $\G$ be a finitely generated group; then $\G$ admits a presentation adjusted to $\op{ab}$.
\end{lemma}

\begin{proof}
Since $\Gamma$ is finitely generated there exists an epimorphism $\pi \colon \langle y_1,\dots,y_r\rangle \to \Gamma$ (where $\langle \dots\rangle$ is the free group on the given set). Let $m=b_1(\Gamma)$. We pick $h_1,\dots,h_m\in \Gamma$ such that $\op{ab}(h_1),\dots,\op{ab}(h_m)$ form a basis of  $H_1(\Gamma)/\mbox{Tor}$. Let $\sigma \colon \langle x_1,\dots,x_m\rangle \to \Gamma$ be the unique homomorphism with $\sigma(x_i)=h_i$, $i=1,\dots,m$.
For $j=1,\dots,r$ we pick $w_j\in \langle x_1,\dots,x_m\rangle $ with $\op{ab}(\sigma(w_j))=\op{ab}(\pi(y_j))$. Let  $\tau \colon \langle x_{m+1},\dots,x_{m+r}\rangle \to \Gamma$ be the unique homomorphism with $\tau(x_{m+j})=\pi(y_j) \cdot \sigma(w_j)^{-1}$. Note that the collection $\{\sigma(x_i),\tau(x_{m+j}), i=1,\dots,n, \,\ j=1,\dots,r\}$ is a generating set for $\G$. It follows that the epimorphism $\sigma*\tau\colon \langle x_1,\dots,x_{m+n}\rangle \to \Gamma$ defines a presentation of $\G$ with the desired properties.
\end{proof}

Note that, given a presentation adjusted to $\op{ab}$, we can and we will associate a basis $e_i, i = 1,\dots,m$ for $\mbox{Hom}(\G,\Z) = H^{1}(\G;\Z)$ via $e_{i}(h_{j}) = \delta_{ij}, i,j = 1,\dots,m$.

Now we are in position to prove our main result.

\begin{theorem} \label{thm:mainthm}
Let $G$ be a finitely generated group that can be written as group extension  \begin{equation} \label{eq:ses2}  1 \longrightarrow K \stackrel{i}{\longrightarrow}  G \stackrel{f}{\longrightarrow} \G \longrightarrow 1 \end{equation} where $K$ is a finitely generated group. Assume furthermore that $b_1(G) > b_1(\G) = m > 0$. Then $G$ algebraically fibers. 
\end{theorem}

\begin{proof} As $G$ is finitely generated, so is $\G$. We choose a presentation of $\G$ adjusted to $\op{ab}$ and correspondingly, denoting $n = m + r$ we have an epimorphism $F_n \to \G$.  

This epimorphism induces a diagram
\begin{equation} \label{eq:epi} \xymatrix{
1\ar[r]&
K \ar[r]\ar[d] &
\Pi \ar[d] \ar[r] & F_n \ar[d] \ar[r]&1\\
 1\ar[r]& K 
 \ar[r]& G \ar[r]  &
\G \ar[r]  &1 } \end{equation}
where all vertical maps are epimorphisms. 

Observe that we can write $\Pi$ as amalgamated product \[ \Pi = \Pi_{1} \ast_K \Pi_2 \ast_K \dots \ast_K \Pi_n \] where each $\Pi_i$ has the form of HNN extension  $\Pi_i = \langle K,s_i|s_i k s_i^{-1} = f_i(k) \rangle$ for some automorphism $f_i\colon K \to K$. There are many ways to see that explicitly, for instance by applying to $\Pi$ the method to write a presentation of  group extensions, as described e.g. in \cite[Section 10.2]{Jo97}.

The mapping torus structure of $\Pi_i$ guarantees the existence of a sequence 
\[  1 \longrightarrow K \longrightarrow \Pi_i \stackrel{\a_i}{\longrightarrow} \Z \longrightarrow 1. \] 
As they vanish on $K$, each of these maps $\a_i\colon \Pi_i \to \Z$ extends to an epimorphism (that we denote with the same symbol) $\a_{i}\colon \Pi \to \Z$. The first $m$ of these elements, identified with a primitive elements of  $H^1(\Pi;\Z)$, can be thought of as pull--back of the classes $e_i  \in H^1(\G;\Z)$ built from the adjusted presentation of $\G$ under the monomorphism $H^1(\G;\Z) \to  H^{1}(F_n;\Z) \to H^{1}(\Pi;\Z)$; by commutativity of the diagram in (\ref{eq:epi}) we can also view $\a_i$ as pull-back of the class $a_i = f^{*}e_i \in H^{1}(G;\Z)$. The remaining $r$ classes $\a_i \in H^1(\Pi;\Z)$, 
that by construction do not pull back from $\G$ nor $G$, will play little role in what follows.  

As the action of $F_{n}$ on  $H^{1}(K;\Z)$ factors through $F_n \to \G$, the assumption  that $b_1(G) > b_1(\G)$ entails that the $\a_i$'s do not generate the entire group $H^1(\Pi;\Z)$, or equivalently the image of $H_1(K;\R)_{F_n} \to H_1(\Pi;\R)$ is nonzero. In particular we can assume the existence of a class $\gamma \in H^{1}(\Pi;\Z)$, pull back of a class $c \in H^1(G;\Z)$ which is not in the image of $H^{1}(\G;\Z) \to H^{1}(G;\Z)$. Restricted to each $\Pi_{i}$, the class 
$\gamma_{i} = \gamma|_{\Pi_{i}}\colon \Pi_{i} \to \Z$ is not a multiple of $\a_i$, as it does not vanish on each $K \unlhd \Pi_{i} \leq \Pi$: in fact, for each $i$, the image of $\gamma_i$ under the map $H^{1}(\Pi_i;\Z) \to H^{1}(K;\Z)^{\Z}$ is nontrivial, as the inclusion $\Z \to F_n$ induces an epimorphism $H^{1}(K;\Z)^{\Z} \to H^{1}(K;\Z)^{F_n}$ and $\gamma$  has nontrivial image in the latter, as the action of $F_{n}$ on  $H^{1}(K;\Z)$ factors through $F_n \to \G$.

 Consider now the partial amalgamation \[ \Pi_A := \Pi_{1} \ast_K \Pi_2 \ast_K \dots \ast_K \Pi_m \leq \Pi \] of the first $m$ factors of $\Pi$. 
For each factor $\Pi_i,  i = 1,\dots,m$, we define the classes \[ \b_i := \a_i + \mu \gamma_i\colon  \Pi_i \longrightarrow \R \] where $\mu$ is a rational number.

As these classes agree on the amalgamating subgroups $K$, they define a class $\b_A\colon \Pi_A \to \R$ by the condition $\b_A|_{\Pi_i} = \b_i$.  Without loss of generality, as $\mu$ is rational we can assume that the resulting $\b_A\colon \Pi_{A} \to \R$ has values in $\Z$, and by construction 
$\b_{A}$ is the pull back of a class $b_{A} \in H^1(G;\Z)$. 

At this point we want to use Lemma \ref{lem:streb} to show that, choosing $\mu$ small enough, the character determined by $\b_A$ is in 
$\Sigma^{1}(\Pi_A)$.  First, as $[\a_i]$ is in $\Sigma^{1}(\Pi_i)$ and the latter is open in $S(\Pi_i)$,    we can assume that for $\mu$ small enough each $[\b_i]$ is still in $\Sigma^{1}(\Pi_i)$. Next, we claim that $[\b]$ is nontrivial on each amalgamating subgroup $K$. In fact on each $\Pi_i$ we have a diagram 
\[ \xymatrix@=9pt{ 
& & & 1  \ar[dr]  & &  &  &  \\
 & & & & K \ar[dr]  \ar[drrr]^{\b_i} & & & & \\ 
& 1\ar[rr] &  & \mbox{ker}  \hspace{1pt}  \b_i  \ar[rr] & & \Pi_i \ar[rr]_{\b_i} \ar[dr]_{ \a_i}  &  & \Z   \ar [rr] & & 0 \\ 
 &  &  & & & & \Z \ar[dr] & & \\
& & &   & &  &   & 0
  } \] 
By contradiction, if $\b_i(K) = 0$ we would have $K \unlhd \mbox{ker}  \hspace{1pt} \b_{i}$ and there would be a short exact sequence \[ 1 \longrightarrow \mbox{ker}  \hspace{1pt} \b_{i}/ K \longrightarrow \Pi_i/K \longrightarrow \Pi_i/\mbox{ker}  \hspace{1pt} \b_{i} \longrightarrow 1; \]
as the latter two groups are infinite cyclic, a surjection is an isomorphism from which it would follow that $K \cong \mbox{ker}  \hspace{1pt}  \b_{i}$.
But this would imply that $\b_i$ is a multiple of $\a_i$, and then so would $\gamma_{i}$, contrary to our assumption.
Next we can invoke (inductively) Lemma \ref{lem:streb}, which asserts that $[\b_A] \in \Sigma^{1}(\Pi_A)$ as long as $[\b_A|_{\Pi_i}] = [\b_i] \in \Sigma^{1}(\Pi_i)$ and $[\b_A]$ is nontrivial on each amalgamating subgroup $K$. Therefore $[\b_A] \in \Sigma^{1}(\Pi_A)$. 

The argument above can be applied \textit{verbatim} for the class $-\beta_A \in H^{1}(\Pi_A;\Z)$; the key point is that, by construction, also the character $[-\a_i] \in \S^{1}(\Pi_i)$. Summing up, both  $[\b_A],[-\b_A] \in \S^1(\Pi_A)$, hence $\mbox{ker} \hspace{1pt} \b_A$ is finitely generated. 

In the case where we can choose $r = 0$ (e.g.\ when $\G$ is a surface group, or the free group itself) we have $\Pi_A = \Pi$ and we would be (essentially) done. But if $r > 0$ we have another hurdle, namely choosing an extension of $\b_A$ to $\Pi$. Obviously, we could follow the pattern above and define $\b_i := \a_i + \mu \gamma_i$ also for the remaining factors. This would give us an algebraic fibration of $\Pi$, but the fibration would not descend to $G$: the classes $\a_i, i = m+1,\dots,n$ are not pull-back of classes on $G$!

The correct way to proceed is, in some sense, counterintuitive. In fact, on the partial amalgamation of the last $r$ factors of $\Pi$ \[ \Pi_B := \Pi_{m+1} \ast_K \Pi_2 \ast_K \dots \ast_K \Pi_n \leq \Pi, \] 
we define $\b_B := \mu \gamma|_{\Pi_B}$. As $\mu$ is rational, we can assume (after simultaneous rescaling if necessary) that $\b_{A}\colon \Pi_{A} \to \Z$, $\b_{B}\colon \Pi_{B} \to \Z$ are epimorphisms that satisfy $\b_{A}|_K = \mu \gamma|_{K} = \b_{B}|_K$, as on $K$ the $\a_i$ vanish, hence give a well--defined epimorphism \[ \b\colon \Pi = \Pi_A *_K \Pi_B \to \Z. \] 
This epimorphism factorizes thorough $\Pi \to G$, as by construction it is a sum of classes that do. Note that
(and this is the key property) on $\Pi_{i}, i = m+1,\dots,n$ we have $\b_{i}(s_i) = 0$: in fact for $i = m+1,\dots,n$ the epimorphism $\gamma_{i}\colon \Pi_i \to \Z$ sends the stable letter $s_i$ to $0 \in \Z$, as the image of $s_i$ in $\G$ belongs to $\mbox{ker} \op{ab}$.

We can now apply inductively Lemma \ref{lem:nonfib} to $\b$. As first step, consider $\Pi_{A} \ast_{K}  \Pi_{m+1}$; $[\b_A] \in \S^{1}(\Pi_A)$ and $[\b_{K}] \in S(K)$ while $[\b_{m+1}] (s_{m+1}) = 0$; Lemma \ref{lem:nonfib} gives that $[\b_{\Pi_A \ast_{K}  \Pi_{m+1}}] \in \S^{1}(\Pi_{A} \ast_{K}  \Pi_{m+1})$, and we can then repeat the process for the remaining factors $\Pi_{i}$. 

As before, we can repeat this argument for the class $-\b \in H^{1}(\Pi;\Z)$, to deduce that both $[\b],[-\b] \in \S^{1}(\Pi)$, hence $\mbox{ker} \hspace{1pt} \b$ is finitely generated.

Finally, as $\b\colon \Pi \to \Z$ factorizes through $\Pi \to G$, we have the diagram

\begin{equation} \xymatrix{
1\ar[r]&
\mbox{ker}  \hspace{1pt}  \b \ar[r]\ar[d] &
\Pi \ar[d] \ar[r]^{\b} & \Z \ar[d]^{\cong} \ar[d] \ar[r] & 0 \\
1\ar[r]&
\mbox{ker}  \hspace{1pt}  b \ar[r]\ar[d] &
G \ar[d] \ar[r]^{b} & \Z  \ar[r] & 0 \\
 & 1 & 1  &  & } \end{equation}
where $b \in H^{1}(G;\Z)$ which entails that the finitely generated  group $\mbox{ker}  \hspace{1pt}  \b$ surjects onto $\mbox{ker}  \hspace{1pt}  b$, which is therefore finitely generated as well. By suitably scaling $b \in H^{1}(G;\Z)$ we can assume it to be primitive, hence it represents an algebraic fibration. 
\end{proof}

\begin{remarks} 
\begin{enumerate} \item The reader may have noticed that in the proof of Theorem \ref{thm:main} 
we actually need to make use of only one (or any subcollection) of the terms $\a_{i} \in H^{1}(\Pi;\Z), i = 1,\dots,m$ (say $i=1$) and not all simultaneously. This follows by applying Lemma \ref{lem:nonfib} inductively to the class $\beta_1 := \a_1 + \mu \gamma_1$ starting with $\Pi_{1} *_{K} \Pi_2$ and repeating the argument until exhausting $\Pi$. This bypasses the use of Lemma \ref{lem:streb}.
\item In Version 2 of \cite{KrWa19} the authors have now provided a concurrent proof for the case $n = \mbox{rank}(\G) = b_1(\G)$ of Theorem \ref{thm:main}.
\item Note that this theorem holds true for simple reasons when $\G$ algebraically fibers.
\end{enumerate}
\end{remarks}

The result above has some consequences in the study of the coherence of the fundamental group of a surface bundle over a surface $F \hookrightarrow X \to B$. In fact, using Theorem \ref{thm:main} we can give a proof of a corollary, that extends to surface bundles with (virtual) Albanese dimension $2$ the results of \cite{FV19}:

\begin{corollary}  \label{cor} Let $F \hookrightarrow X \to B$ be a surface bundle over a surface with both base and fiber of genus greater than $1$. If its Albanese dimension $a(X) = 2$, then the first  BNS-invariant and the second BNSR-invariant of its fundamental group $G = \pi_1(X)$ satisfy the relation 
\[  \S^2(G) \subsetneq \S^1(G) \subsetneq S(G), \] and $G$ is noncoherent. If $va(X) = 2$, $G$ is noncoherent.
\end{corollary}`
\begin{proof} Let $X$ have Albanese dimension $2$; by Theorem~\ref{thm:mainthm}, there exists  an epimorphism $b\colon G \to \Z$ with finitely generated kernel $\mbox{ker} \hspace{1pt} b$, hence $[b],[-b] \in \Sigma^1(G)$. By \cite[Theorem 4.5(4)]{Hil02} $\mbox{ker} \hspace{1pt} b$ would have type $FP_2$ if and only if the Euler characteristic $\c(X) = 0$; as here $\chi(X) = (2g(F)-2)(2g(B)-2) > 0$,  $\mbox{ker} \hspace{1pt} b$ is not finitely presented. Therefore, at least one among  $[b],[-b]$ is not in 
$\Sigma^2(G)$.

We want to give a second, and somewhat different, proof of noncoherence, based on the work of \cite{KrWa19} which avoids the use of the BNS invariant of $G$.

Consider any subgroup $F_2 \leq \G$ and the corresponding commutative diagram
\begin{equation} \xymatrix{
1\ar[r]&
K \ar[r]\ar[d] &
\Lambda \ar[d] \ar[r] & F_2 \ar[d] \ar[r]&1\\
 1\ar[r]& K 
 \ar[r]& G \ar[r]  &
\G \ar[r]  &1 } \end{equation} 
with self-explaining notation where all vertical arrows are monomorphisms.  Now, by standard group homology \[ H_{1}(K;\R)_{F_2}  \longrightarrow   H_1(K;\R)_{\G}\] is an epimorphism. As discussed in the introduction, the assumption $a(X) = 2$  is equivalent to the nontriviality of $H_1(K;\R)_{\G}$, and so implies the nontriviality of $H_{1}(K;\R)_{F_2}$. At this point, we invoke \cite[Theorem 4.5]{KrWa19} where they show that if $K$ is a group that does not algebraically fiber (as is our case) a group $\Lambda$ that is $K$--by--$F_2$ and has nontrivial $H_{1}(K;\R)_{F_2}$ is noncoherent. As $\Lambda$ is noncoherent, so is $G$.

If $va(X) = 2$, a finite cover of $X$ will satisfy this property, hence the fundamental group of $X$ will not be coherent as well.

\end{proof}

One can ask whether the techniques of Theorem \ref{thm:main} can be extended to a complete characterization of the BNS invariant $\S^1(G)$, at least for the case of surface bundles over a surface. This appears challenging on several grounds. The first can be appreciated by pointing out the complexity of the situation already in the case of Kodaira fibrations. Delzant~\cite{De10} shows that as long as $a(X) = 2$, the first BNS invariant is the complement of the pull-back of the character spheres of the bases of all irrational pencils $h_i\colon X \to B_i$ with base a hyperbolic orbisurface. (This includes, whenever possible, surface bundle maps with base of genus bigger than $1$.)  From the group-theoretical viewpoint, these correspond to short exact sequences for $G = \pi_1(X)$  of the form 
\begin{equation} \label{eq:penseq} 1 \longrightarrow M_i \longrightarrow  G \stackrel{h_i}{\longrightarrow} C_i \longrightarrow 1 \end{equation}
where $C_i$ is the fundamental group of the hyperbolic orbisurface $B_i$ and $M_i$ is a finitely generated group (see e.g. \cite{Cat03}). As a consequence, $\Sigma^{1}(G)$ is the complement of a finite collection of spheres of codimension at least $2$, pull-back of the character spheres of the orbisurfaces. An example of this phenomenon arises already in the case of doubly fibered Kodaira fibrations (such as Atiyah and Kodaira's original examples). Moreover there exist examples even of Kodaira fibrations which admit also pencils with multiple fibers (see \cite{Br18}).
  But for the case of (non--K\"ahler) surface bundles the situation can be even more complex, as the following construction shows. 

\begin{proposition} There exists a surface bundle $F \hookrightarrow X \stackrel{f}{\rightarrow} B$ with base and fiber of genus greater than $1$ whose fundamental group $G$ admits an epimorphism $g\colon G \to F_2$ such that $H^1(G;\R) = f^{*}H^{1}(\G;\R) \oplus g^{*}H^{1}(F_2;\R)$, where $\G := \pi_1(B)$,  and so that the set of exceptional characters contains two disjoint spheres 
\[  f^{*}(\S^{1}(\G)^{c}) \cup g^{*}(\S^{1}(F_2)^{c})\,\subset\, 
\Sigma^{1}(G)^{c}. \]
Moreover for any epimorphism $h\colon G \to C$ onto the fundamental group of a  hyperbolic orbisurface $C$ such that the kernel is finitely generated we have  $g^{*}(\S^{1}(F_2)^{c})\cap h^*(\S^1(C)^c)=\emptyset$.  \end{proposition}

\begin{proof} We begin with the somewhat lengthy, but otherwise straightforward construction of $X$.

Let $F$ be a surface of genus $2$ and let \[ K := \langle \a_i,\b_i|\prod_{i=1}^2 [\a_i,\b_i]= 1 \rangle \] be a standard presentation of its fundamental group. (Above and in the rest, we will always assume implicitly that $i =  1,2$.) Define an automorphism $\varphi \in \mbox{Aut}(K)$ as follows: let \[ \varphi(\a_i) = \a_i \b_i; \,\ \varphi(\b_i) = \b_i; \] a straightforward calculation shows that $\varphi(\prod_{i=1}^2 [\a_i,\b_i]) = \prod_{i=1}^2 [\a_i,\b_i]$, hence $\varphi$ is well--defined; it is immediate to verify that \[ \varphi^{-1}(\a_i) =  \a_i \b^{-1}_i; \,\ \varphi^{-1}(\b_i) = \b_i \] is similarly well--defined, and a two--sided inverse to $\varphi$. We can lift $\varphi \in \mbox{Aut}(K)$ to a self--diffeomorphism of $F$ (which can be easily proven to be orientation-preserving), that we denote as well by $\varphi\colon F \to F$, which induces the above automorphism on the fundamental group $K$. 
Out of that, we can define a fibered $3$--manifold $N$ as the mapping torus of $\varphi\colon F \to F$. Its fundamental group is given by 
\[ \pi_1(N) = K \rtimes_{\varphi} \Z = \langle \a_i,\b_i,t|\a_i^t = \a_i \b_i, \b_i^t = \b_{i},\prod_{i=1}^2 [\a_i,\b_i] = 1 \rangle. \] 
The $4$--manifold $S^1 \times N$ has the structure of $F$--bundle over $T^2$. We can pick a trivial framing of $F$ in $S^1 \times N$ and in the product bundle $F \times T^2 \to T^2$ to define the fiber sum 
\[ X := S^1 \times N \#_F F \times T^2 = S^1 \times N \setminus F \times D^2  \cup_{F \times \partial D^2} F \times (T^2 \setminus D^2), \] 
choosing any identification of the fibers and an orientation--reversing diffeomorphism of $\partial D^2$  to get a surface bundle $X$  of fiber $F$ over a genus--$2$ surface $B$. It is not difficult to compute the resulting fundamental group using Seifert--Van Kampen Theorem: $S^1 \times N \setminus F \times D^2$ deformation--retracts to a topological bundle over $S^1 \vee S^1$, with monodromy 
$\varphi$ along one factor and trivial monodromy on the other, so that 
\[ \pi_1(S^1 \times N \setminus F \times D^2) = \langle \a_i,\b_i,s,t|\a_i^t = \a_i \b_i, \b_i^t = \b_{i},\a_i^s = \a_i, \b_i^s = \b_{i},\prod_{i=1}^2 [\a_i,\b_i] = 1 \rangle \]
while \[ \pi_1(F \times (T^2 \setminus D^2)) = \langle \a_i,\b_i,x,y|a_i^x=a_i,b_i^x=b_i,a_i^y=a_i,b_i^y=b_i,\prod_{i=1}^2 [\a_i,\b_i] = 1 \rangle. \] Amalgamation identifies the generators of $\pi_1(F)$ according to their symbols, and adds the relation $[s,t][x,y] = 1$, so that $G = \pi_1(X)$ is given by \[ G = \langle \a_i,\b_i,s,t,x,y|\a_i^t = \a_i \b_i, \b_i^t = \b_{i},\a_i^s = \a_i, \b_i^s = \b_{i},\prod_{i=1}^2 [\a_i,\b_i] = 1, [s,t][x,y] = 1 \rangle. \]

We can proceed now to show that $X$ satisfies the properties of the statement. The key strategy, that dates back in this context at least to the work of Johnson (\cite{Jo93}, see also \cite{Cat03,Sa15})  is based on the fact that finitely generated normal subgroups of free groups or hyperbolic orbisurface groups are either trivial or finite index.

Following the usual notation, we have the quotient map $G \stackrel{f}{\longrightarrow} \G$ according to the sequence in (\ref{eq:ses}). Besides that, we can identify a second infinite quotient of $G$: define a map $g\colon G \to F_2$, the free nonabelian group on two generators, by sending $\a_1$ and $\a_2$ to the two generators while all other generators are sent to the trivial element. A straightforward calculation shows that all the relations in the presentation of $G$ are sent to the trivial element, namely $g\colon G \to F_2$ is well--defined. (A long exercise -- that we don't recommend to the reader -- in Reidemeister--Schreier rewriting process shows that $\Lambda := \mbox{ker} \hspace{1pt} g$ surjects to $\Z^{\infty}$, in particular is not finitely generated.) The setting relating the two quotients of $G$ is described in the following diagram. 
\[ \xymatrix@=9pt{ 
& & & 1  \ar[dr]  & &  &  &  \\
 & & & & K \ar[dr]  \ar[drrr]^{g} & & & & \\ 
& 1\ar[rr] &  & \Lambda  \ar[rr] & & G \ar[rr]_{g} \ar[dr]_{ f} &  & F_2 \ar[rr] & & 1\\ 
 &  &  & & & & \G \ar[dr] & & \\
& & &   & &  &   & 1
  } \]
We can see explicitly that $g(K) = F_2$, from which we deduce that $K\Lambda = G$. 
This entails that $f^*H^1(\G;\R) \cap g^*H^1(F_2;\R) = \{0\}$. Now $\mbox{dim\hspace*{1pt}}H^1(G;\R) = 6$, hence there is a direct sum decomposition \begin{equation} \label{eq:dirsum}  H^1(G;\R) = f^*H^1(\G;\R) \oplus g^*H^1(F_2;\R) .\end{equation} 
We obtain therefore from \cite[Corollary~B1.8]{St13} that
 \begin{equation} \label{eq:strebel} f^{*}(\S^{1}(\G)^{c}) \cup g^{*}(\S^{1}(F_2)^{c})\,\subset \,\Sigma^{1}(G)^{c}. \end{equation}
Note that \eqref{eq:dirsum} and \eqref{eq:strebel} imply  that the complement of the BNS invariant of $G$, i.e. the set of exceptional characters, contains at least two spheres, one of codimension $2$ and one of codimension $4$, the latter determined by $g^*H^1(F_2;\R)$.
We claim that no exceptional characters in $g^{*}(\S^{1}(F_2)^{c})$ factorize through a second surface bundle structure nor a pencil--type sequence as in Eq. (\ref{eq:penseq}). 

To prove this claim, assume by contradiction that
 $g^{*}(\S^{1}(F_2)^{c}) \cap h^{*}(\S^{1}(C)^{c}) \neq \emptyset $ where $h\colon G\to C$ is an epimorphism onto a  hyperbolic orbisurface group such that $M:=\mbox{ker} \hspace{1pt} h $ is finitely generated. We will start by showing that this entails that $g\colon G \to F_2$ factorizes through $h\colon G \to C$. Consider the diagram 
\[ \xymatrix@=9pt{ 
& & & 1  \ar[dr]  & &  &  &  \\
 & & & & M \ar[dr]  \ar[drrr]^{g} & & & & \\ 
& 1\ar[rr] &  & \Lambda  \ar[rr] & & G \ar[rr]_{g} \ar[dr]_{ h} &  & F_2 \ar[rr] & & 1\\ 
 &  &  & & & & C \ar[dr] & & \\
& & &   & &  &   & 1
  } \]
Now $g(M) \leq F_2$ is a finitely generated normal subgroup of $F_2$, hence it must be trivial or finite index. If it were finite index, then $M \Lambda \leq G$ would be finite index and $g^*H^1(F_2;\R)  \cap h^*H^1(C;\R)= \{0\}$, which would imply $g^{*}(\S^{1}(F_2)^{c}) \cap h^{*}(\S^{1}(C)^{c}) = \emptyset$. It follows that we must have $g(M) \leq F_2$ trivial. This entails $M \leq \Lambda$, hence $G/M = C$ admits an epimorphism onto $G/\Lambda = F_2$, or phrased otherwise  $g\colon G \to F_2$ factorizes though $h\colon G \to C$. Next, we will show that this factorization is not compatible with having $b_1(G) = 6$. In fact this would yield the diagram
\[ \xymatrix@=9pt{ 
& & & 1  \ar[dr]  & &  &   & 1 \\
 & & & & K \ar[dr]  \ar[rr]^{h} & & C \ar[dr] \ar[ur] & & \\ 
&  &  &  & & G \ar[rr]_{g} \ar[ur]^(0.4){ h} \ar[dr]_{ f} &  & F_2 \ar[rr] & & 1\\ 
 &  &  & & M \ar[ur] & & \G \ar[dr] & & \\
& &  &  1 \ar[ur] & &  &   & 1
  } \] Now $h(K) \leq C$ is a finitely generated normal subgroup of $C$, so again it can be either trivial or finite index. It cannot be trivial, because $g(K) \leq F_2$ is already nontrivial, so $h(K) \leq C$ is finite index. Once again, $KM \leq G$ is finite index. But then we would have $f^*H^1(\G;\R) \cap h^*H^1(C;\R) = \{0\}$. Now an explicit check shows that an orbisurface group $C$ with an $F_2$ quotient must have $b_1(C) \geq 4$, hence  we would have $b_1(G) \geq b_1(\G) + b_1(C) \geq 8$, which violates the condition $b_1(G) = 6$.   \end{proof}

\begin{remark} Note that the information on the BNS invariant of $G = \pi_1(X)$ contained in the above Proposition informs us already that $G$ cannot be a K\"ahler group, as it violates the conclusions of \cite{De10}. Alternatively, we could use Catanese's version of Castelnuovo--de Franchis Theorem (see \cite{Cat91}) to argue that if $G$ were K\"ahler, the map $g$ would have to factor through a map to a orbisurface group of genus at least $2$. But this would require, again, that $b_{1}(G) \geq 8$. \end{remark}


\end{document}